\documentclass[12pt,leqno]{article}
\usepackage[active]{srcltx}

\usepackage{amsfonts,latexsym,amsmath,amssymb,amsthm}
\usepackage{fullpage}
\newtheorem{theorem}{Theorem}[section]
\newtheorem{lemma}[theorem]{Lemma}

\newtheorem{proposition}[theorem]{Proposition}

\newtheorem{problem}[theorem]{Problem}
\theoremstyle{definition}
\newtheorem{definition}[theorem]{Definition}

\def\R{{\mathbb R}}
\def\N{{\mathbb N}}

\theoremstyle{remark}

\newtheorem*{note*}{Note}

\numberwithin{equation}{section}

\makeatletter
\newcommand{\rank}{\mathop{\operator@font rank}}
\newcommand{\conv}{\mathop{\operator@font conv}}
\newcommand{\vol}{\mathop{\operator@font vol}}
\newcommand{\onetagright}{\tagsleft@false}
\makeatother

\renewcommand{\epsilon}{\varepsilon}

\usepackage{color}

\newtheorem{note}{Remark} 
\def\note#1{\ifvmode\leavevmode\fi\vadjust{\vbox to0pt{\vss
 \hbox to 0pt{\hskip\hsize\hskip1em
\vbox{\hsize2.5cm\small\raggedright\pretolerance10000
 \noindent #1\hfill}\hss}\vbox to8pt{\vfil}\vss}}}

\usepackage{hyperref}

\begin{document}

\title{\bf On polynomially integrable convex bodies}

\medskip

\author{A.~Koldobsky, A.~Merkurjev, and V.~Yaskin}

\date{}

\maketitle

\begin{abstract} An infinitely smooth convex body in
  $\mathbb R^n$ is  called polynomially integrable of degree $N$ if its parallel section
  functions are polynomials of degree $N$. We prove that the only smooth
  convex bodies with this property  in odd dimensions are ellipsoids,
  if $N\ge n-1$.  This is in contrast with the case of even dimensions
  and the case of odd dimensions with $N<n-1$,
  where such bodies do not exist,  as it was recently shown by Agranovsky.
\end{abstract}

\section{Introduction}

Let $K$ be an  infinitely smooth  convex body in $\R^n.$ The parallel
section function of $K$ in the direction $\xi\in S^{n-1}$
is defined by
$$A_{K,\xi}(t)=\mathrm{vol}_{n-1}(K\cap \{(x,\xi)=t\})=\int_{(x,\xi)=t} \chi_K(x) dx,
\quad t\in \R,$$
where $\chi_K$ is the indicator function of $K$, and $(x,\xi)$ is the scalar product in $\R^n$.

It is clear that if $B$ is the Euclidean ball of radius $r$ centered at the origin,
then $$A_{B,\xi}(t)= c_n (r^2 - t^2)^{(n-1)/2},$$ for $|t|\le r$. In
particular, if $n$ is odd then   the
parallel section function of $B$ is a polynomial  for every $\xi\in
S^{n-1}$. This property also holds for ellipsoids.

\begin{definition}\label{pi} A convex body $K$ (or more generally, a bounded domain) in $\R^n$ is called {\it polynomially
    integrable} (of degree $N$) if
\begin{equation}\label{pol-int}
A_{K,\xi}(t)=\sum_{k=0}^N a_k(\xi)\ t^k
\end{equation}
for some integer $N$, all $\xi\in S^{n-1}$ and all $t$ for which
the set $K\cap \{x: (x,\xi)=t\}$ is non-empty. Here, $a_k$ are functions on the sphere. We assume that
the function $a_N$ is not identically zero.
\end{definition}

This concept was introduced by Agranovsky in \cite{A}. He also
established a number of properties of such bodies.
In particular, he showed that there are no bounded polynomially integrable domains with smooth boundaries in Euclidean spaces of even
dimensions. In odd dimensions he proved that
polynomially integrable bounded domains with smooth boundaries are convex, and that there are no polynomially integrable bounded domains in
$\mathbb R^n$
with smooth boundaries of degree strictly less than
$n-1$, while every such body with degree $n-1$ is an ellipsoid.
For polynomially integrable domains of higher degrees Agranovsky asks the following.

\begin{problem}\label{ellipsoid}  Is it true that in the odd-dimensional
  space  the only polynomially
  integrable domains
are ellipsoids?
\end{problem}

Problems of this kind go back
to Newton \cite{N}.  Consider the volume of the ``halves'' of
the body cut off by the hyperplane $(x,\xi)=t$, that is
$V^+_{K,\xi}(t) = \int_t^\infty A_{K,\xi}(z) dz$ and $V^-_{K,\xi}(t) =
\int_{-\infty}^t A_{K,\xi}(z) dz$. A body $K$ is called algebraically integrable if
there is a polynomial $F$ such that
$F(\xi_1,\dots,\xi_n,t,V^\pm_{K,\xi}(t))=0$ for every choice of parameters $\xi$ and $t$. Newton showed that in $\mathbb R^2$  there
are no algebraically integrable convex bodies with infinitely smooth
boundaries.  Arnold asked for extensions of Newton's result to other
dimensions and general domains; see problems 1987-14, 1988-13, and
1990-27 in \cite{Arnold}.
 Vassiliev  \cite{V} generalized Newton's
result by showing that there are
no algebraically integrable bounded domains with infinitely smooth boundary in
$\R^n$ for even $n$.

Although  Agranovsky's question stated above (Problem \ref{ellipsoid})
has a lot in common with these types of
problems,  yet it is different. While for each $\xi$
the parallel section function is assumed to be polynomial, there are
no conditions on the coefficients $a_k$. Thus, it is not
a particular case of Arnold's question.
\smallbreak
Here we consider the following  question.

\begin{problem} \label{deriv} Suppose that $K$ is an
  infinitely smooth convex body in $\R^n,$ containing the origin in its interior,
and suppose there exists $N\in \N$ such that for every even integer $m\ge N$ and every $\xi\in S^{n-1}$
we have
$$A_{K,\xi}^{(m)}(0)=0.$$
Is it true that the body $K$ is necessarily an
ellipsoid if $n$ is odd and  $N\ge n$, and that such
bodies do not exist in all other cases?
\end{problem}

In this paper we give an affirmative  answer to Problem \ref{deriv},
thus solving Problem \ref{ellipsoid}, since convexity follows from the condition (\ref{pol-int}), as it was shown in \cite[Theorem 5]{A}. Also the condition that $K$ contains the origin is not restrictive because
polynomial integrability is invariant with respect to shifts. Note that
the number $N$ in Problem \ref{deriv} differs by one from that in
Definition \ref{pi}.

\section{Preliminaries}

We say that $K$ is a {\it star body} in $\R^n$ if it is compact, star-shaped with respect to the origin, and
the {\it Minkowski functional} of $K$  defined by
\begin{align*}
\|x\|_K=\min \{a\ge 0: x \in aK \}, \qquad x\in \mathbb R^n.
\end{align*}
is a positive continuous function on $\mathbb R^n\setminus\{0\}.$

$K$ is said to be infinitely smooth if its Minkowski functional is a
$C^\infty$-function on $\mathbb R^n\setminus \{0\}$. $K$ is origin-symmetric if $K=-K$.

Most of the time we will assume that $K$ is a
convex body in $\R^n$, i.e. a compact convex set  with non-empty
interior.

In this paper we will be working with the Fourier transform of
distributions; we refer to \cite{K-book} and \cite{GYY} for details. Let $\mathcal{S}(\mathbb{R}^n)$  be the Schwartz space of infinitely differentiable
rapidly decreasing functions on $\mathbb R^n$, called test functions. The Fourier transform of $\phi\in\mathcal{S}(\mathbb{R}^n)$ is a test function $\widehat{\phi}$ defined by
\begin{align*}
 \widehat{\phi}(x) = \int_{\mathbb{R}^n} \phi(y) e^{-i(x,y)} \, dy, \qquad x\in\mathbb{R}^n .
\end{align*}
By $\mathcal S'(\mathbb R^n)$ we denote the space of continuous linear
functionals on $\mathcal S(\mathbb R^n)$. Elements of
$\mathcal{S}'(\mathbb{R}^n)$ are referred to as distributions. We
write $\langle f,\phi\rangle$ for the action of $f\in\mathcal{S}'(\mathbb{R}^n)$ on a test function $\phi$.

The Fourier transform of $f$ is a distribution $\widehat{f}$ defined by
\begin{align*}
\langle \widehat{f},\phi\rangle = \langle f, \widehat{\phi}\rangle , \qquad \forall \phi\in\mathcal{S}(\mathbb{R}^n).
\end{align*}
If a distribution $f$ is supported in the origin, i.e. $\langle f,\phi\rangle=0$ for any test function $\phi$
supported in $\mathbb{R}^n\setminus \{0\},$ then $f$ is a finite sum of derivatives of the delta function,
and its Fourier transform $\widehat{f}$ is a polynomial.

If $K\subset \R^n$ is a  convex body, containing the origin in its interior, and $p>-n$, then
$\|\cdot\|_K^p$ is a locally integrable function on $\R^n$ and
therefore can be thought of as a distribution that acts on test
functions by integration,
$$\langle \|\cdot\|_K^p, \phi\rangle  = \int_{\mathbb R^n} \|x\|_K^p
\phi(x)\, dx, \quad \forall \phi\in \mathcal{S}(\R^n).$$

If $K$ is an infinitely smooth  convex body, containing the origin in its interior, $p>-n$,
and $p\ne 0$, then the Fourier transform of $ \|\cdot\|_K^p$ is given
by a homogeneous of degree $-n-p$ continuous function on $\mathbb
R^n\setminus \{0\}$; see \cite[Lemma 3.16]{K-book}. We write $(\|x\|_K^p)^\wedge(\xi)$
meaning the value of this continuous function at the point $\xi\in S^{n-1}.$
 Moreover, this function can be
computed in terms of fractional derivatives of the parallel section
function of $K$.

Let  $h$ be an integrable
function on $\R$ that is $C^\infty$-smooth in a neighborhood of the
origin. The {\it fractional derivative}  of the function
$h$ of order $q\in \mathbb C$ at zero is defined by
$$h^{(q)}(0) = \langle \frac{t_+^{-1-q}}{\Gamma(-q)}, h(t)\rangle,$$
where $t_+=\max \{0,t\}$.

In particular, if $q$ is not an integer and $-1< \Re q <m$ for some
integer $m$, then

$$h^{(q)}(0)=\frac{1}{\Gamma(-q)} \int_0^1 t^{-1-q} \Big(h(t)-h(0)-\cdots -h^{(m-1)}(0) \frac{t^{m-1}}{(m-1)!}\Big) dt$$    $$ +\frac{1}{\Gamma(-q)}
\int_1^\infty t^{-1-q} h(t)
dt+\frac{1}{\Gamma(-q)}\sum_{k=0}^{m-1}\frac{h^{(k)}(0)}{k!(k-q)}.$$

If  $k\ge 0$ is an integer,  the fractional derivative of the order
$k$ is given as the limit of
the latter expression as $q\to k$. That is
$$h^{(k)}(0) = (-1)^k \frac{d^k}{dt^k} h(t)\Big|_{t=0},$$
i.e. fractional derivatives of integral orders coincide up to a sign with ordinary derivatives.
Note that  $  h^{(q)}(0)$ is an entire function of the variable $q\in
\mathbb{C}.$

If $K$ is an infinitely smooth  convex body, then  $A_{K,\xi}$ is infinitely smooth in a neighborhood of
$t=0$ which is uniform with respect to $\xi\in S^{n-1}$; see  \cite[Lemma
2.4]{K-book} (note that the proof there works without the symmetry assumption). In \cite{RY} (see the proof of Theorem 1.2 there) the following formula is proved:
\begin{eqnarray}\label{Aq-nonsymm} A_{K,\xi}^{(q)}(0) &=& {\cos (q\pi/2)\over {2\pi (n-1-q)  }} \left( \|x\|_K^{-n+1+q} + \|-x\|_K^{-n+1+q}\right)^\wedge(\xi) \\
 & &- {i\sin (q\pi/2)\over {2\pi(n-1-q)  }} \left( \|x\|_K^{-n+1+q} - \|-x\|_K^{-n+1+q} \right)^\wedge(\xi).\nonumber
\end{eqnarray}
In \cite{RY} this formula is stated only for $-1<q<n-1$, but a
standard analytic continuation argument yields a larger range of $q$, namely $q\in (-1,\infty)$, $q\ne n-1$.
In the symmetric case formula (\ref{Aq-nonsymm}) was established in \cite{GKS}, see also
\cite[Th 3.18]{K-book} and a different proof in \cite{BFM}.

In particular, for integers we get the usual derivatives. If $k \ge 0$ is an even integer, $k\ne n-1$, then
\begin{equation}\label{A^(k-even)}
A_{K,\,\xi}^{(k)}(0) =    \frac{(-1)^{k/2}}{2\pi(n-k-1)} \Big(\Vert x\Vert_K^{-n+1+k}+\Vert - x\Vert_K^{-n+1+k}
	 \Big)^\wedge (\xi),
\end{equation}
and if $k> 0$ is an odd integer, $k\ne n-1$, then
\begin{equation}\label{A^(k-odd)}
A_{K,\,\xi}^{(k)}(0) =    \frac{i(-1)^{(k-1)/2}}{2\pi(n-k-1)} \Big(\Vert x\Vert_K^{-n+1+k} - \Vert - x\Vert_K^{-n+1+k}
	 \Big)^\wedge (\xi),
\end{equation}

A formula for the derivative of the order $n-1$ can be obtained precisely as in
\cite{KKYY} (where this was done for origin-symmetric bodies) by taking the limit in (\ref{Aq-nonsymm}) as $q\to
n-1$. In particular, if $n$ is odd, then
\begin{equation}\label{A^(n-1)oddn}
A_{K,\,\xi}^{(n-1)}(0) =   \frac{ (-1)^{(n-1)/2}}{2\pi} \Big(\ln\Vert x\Vert_K + \ln\Vert - x\Vert_K
	 \Big)^\wedge (\xi),
\end{equation}
for every $\xi\in S^{n-1}$.

\section{Main results}

First we prove a version of Agranovsky's result in the setting of Problem \ref{deriv}.

\begin{proposition}
There are no infinitely smooth convex bodies in $\R^n$ for even $n$ satisfying the condition of Problem \ref{deriv}.
\end{proposition}
\begin{proof}
Let $K$ be an infinitely smooth   convex body in $\R^n$
satisfying the condition of Problem \ref{deriv}.
Take an even $m\ge \max\{N,n\}$.  Then using (\ref{A^(k-even)}) we have
$$
 \Big(\Vert x\Vert_K^{-n+1+m}+\Vert - x\Vert_K^{-n+1+m}
	 \Big)^\wedge (\xi)=   { (-1)^{m/2}}{2\pi(n-m-1)}
         A_{K,\,\xi}^{(m)}(0) = 0,
$$
for every $\xi\in S^{n-1}$.

Thus the Fourier transform of $ f(x)= \|x\|_K^{-n+1+m}+\|-x\|_K^{-n+1+m}$ is zero outside of
the origin, implying that $  f(x)$ can only be a
polynomial. This polynomial has to be even, since the
function $  f(x)$  is even. On the other hand, since
$-n+1+m$ is an odd number, $  f(x)$  has to be an odd
polynomial. Thus $f(x)$   is zero for all $x\in \mathbb
R^n$, which is impossible.
\end{proof}

Note that in the proof we needed only one derivative of the parallel section
function.

\begin{proposition} For odd $n,$ there are no smooth bodies in $\R^n$ satisfying the condition of Problem \ref{deriv}
with $N<n.$
\end{proposition}
\begin{proof} Suppose such a body $K$ exists. If $N\le n-3$, we let $m=n-3$  so that $m$ is an even integer.
The distribution $(\|x\|_K^{-n+m+1}+\|-x\|_K^{-n+m+1})^\wedge=(\|x\|_K^{-2}+\|-x\|_K^{-2})^\wedge$ is an extension of
a continuous function on the sphere to a homogeneous function of degree $-n+2$ on the whole of $\R^n.$
Since $-n+2>-n,$   this function is locally integrable, and as a distribution
acts on test functions by integration. On the other hand, by
(\ref{A^(k-even)}) and the hypothesis of the proposition, this homogeneous
function is equal to zero on the sphere. So the distribution $(\|x\|_K^{-2}+\|-x\|_K^{-2})^\wedge$ is equal to zero
on all test functions, and $\|x\|_K^{-2}+\|-x\|_K^{-2}=0$ everywhere, which is impossible.

If   $N= n-2$ or $N=n-1$, the argument is similar. We let $m=n-1$  and
use   formula (\ref{A^(n-1)oddn}) to get
 $$ \Big(\ln\Vert x\Vert_K+\ln\Vert - x\Vert_K
	 \Big)^\wedge (\xi)=0,$$
for every $\xi\in S^{n-1}$.

Since $ \left( \ln   \|x\|_K + \ln\Vert - x\Vert_K \right)^\wedge$ is homogeneous of degree
$-n$, this Fourier transform can only be a multiple of the delta
function supported at the origin. So
 $\ln   \|x\|_K + \ln\Vert - x\Vert_K$ must be a constant for all $x\in\mathbb R^n$, which
 is impossible.

\end{proof}

For the reader's convenience we will first settle the case $N=n$ in odd dimensions, since the argument here is  much simpler than that for $N>n$.

\begin{proposition} Let $K$ be an infinitely smooth convex body in $\R^n$, $n$ is odd,   satisfying the condition of Problem \ref{deriv}
with $N=n.$ Then $K$ is an ellipsoid.
\end{proposition}

\begin{proof}
Using formulas (\ref{A^(k-odd)}) and (\ref{A^(k-even)}) with $k=n$ and $k=n+1$ correspondingly, we get

$$  \left( \|x\|_K  - \|-x\|_K  \right)^\wedge(\xi)=0$$
and
$$  \left( \|x\|_K^{2} + \|-x\|_K^{2}\right)^\wedge(\xi)=0,$$
for all $x \in S^{n-1}$. This implies that
$ \|x\|_K  - \|-x\|_K = P(x)$ and $\|x\|_K^{2} + \|-x\|_K^{2} = Q(x)$ for some homogeneous polynomials $P$ and $Q$ of degrees 1 and 2 correspondingly. Solving this system of two equations, we get
$$\|x\|_K = \frac{1}{2}\left(P(x) +\sqrt{2Q(x)- P^2(x)}\right).$$
To see that this defines an ellipsoid, we let $x$ be a point on the boundary of $K$. Then $\|x\|_K =1$ and therefore
 $$2 -  P(x) = \sqrt{2Q(x)- P^2(x)} .$$
 Squaring both sides, we get an equation of a quadric surface. This can only be the surface of an ellipsoid, since the boundary of $K$ is a complete bounded surface.

\end{proof}


Now we will treat the case $N>n$. We will need the following auxiliary results.

\begin{lemma}\label{lem:product}
Let $L$ be an origin-symmetric star body in $\mathbb R^n$. Suppose that for some positive even $k$ there are homogeneous polynomials $P$ and $Q$   so that $\|x\|_L^k = P(x)$ and $\|x\|_L^{k+2} = Q(x)$  for all $x\in \mathbb R^n$. Then $L$ is an ellipsoid.
\end{lemma}
 \begin{proof} By the hypothesis of the lemma we have  $(P(x))^{k+2} = (Q(x))^{k}$ for all $x$.  Now let us consider any two-dimensional subspace $H$ of $\mathbb R^n$. The restrictions of $P$ and $Q$ to $H$ are again homogeneous polynomials of degrees $k$ and $k+2$ correspondingly. Abusing notation, we will denote these restrictions by $P(u,v)$ and $Q(u,v)$, where $(u,v)\in \mathbb R^2$. Thus we have $ (P(u,v))^{k+2} = (Q(u,v))^{k}$ for all $(u,v)\in \mathbb R^2$.  Since both $P$ and $Q$ are homogeneous, the latter is equivalent to
$$ (P(u,1))^{k+2} = (Q(u,1))^{k}, \qquad \forall u\in \mathbb R.$$
We have the equality of two polynomials of the real variable $u$,
therefore these polynomials are equal for all $u \in \mathbb C$. Let $u_0$ be a complex root of $P(u,1)$ of multiplicity $\alpha\le k$. Then $u_0$ is also a root of $Q(u,1)$ of some multiplicity $\beta\le k+2$. Hence we have
$$\alpha (k+2) = \beta k,$$
$$\frac{\alpha}{\beta} = \frac{k}{k+2}.$$
Recall that $k$ is even, say $k = 2 l$, $l \in \mathbb N$.
Thus $$\frac{\alpha}{\beta} = \frac{l}{l+1}.$$
Since $l$ and $l+1$ are co-prime, there are only two possibilities for $\alpha$ and $\beta$: either $\alpha=l$, $\beta = l+1$, or $\alpha = 2l$, $\beta = 2\l + 2$. The latter is impossible since it implies that
$$ \|(u,v) \|_{L\cap H}^k = P(u,v) = c (u - v u_0)^k,$$
for some constant $c$.
So the remaining possibility is that $P(u,1)$ has two complex roots, say $a$ and $b$ of multiplicity $l$. Therefore,
\begin{eqnarray*}\|(u,1) \|_{L\cap H}^k = P(u,1)& =& c [(u -   a)(u -
  b)]^l \\ & =&  c [u^2 - (a+b) u   + ab ]^l = c [ u^{2l} - (a+b) u ^{2l -1}
  + ab u^{2l -2} + \cdots].
\end{eqnarray*}
Since the restriction of this polynomial to $\mathbb R$ has real coefficients, it follows that $a+b$ and $ab$ are real numbers. Since $a$ and $b$ cannot be real, we conclude that they are complex conjugates of each other. Therefore,    $\|(u,v) \|_{K\cap H}^2 =  \bar c [u^2 - (a+b) u v  + ab v^2 ]$ is a nondegenerate quadratic form. Thus $L\cap H$ is an ellipse. Since every 2-dimensional central section of $L$ is an ellipse, $L$ has to be an ellipsoid. The latter is a consequence of the Jordan - von Neumann characterization of inner product spaces by the parallelogram equality;
see \cite{JN}.
 \end{proof}


We will use the following notation.  $A$ is the ring of polynomial real functions on $\R^n$, $F$  is the field of rational real functions, i.e.,
functions of the form $g/h$, where $g$ and $h$ are in $A$. The ring $A$ is a factorial ring, i.e.,
every nonzero polynomial factors uniquely into a product of irreducible polynomials (see \cite[Chapter IV, Corollary 2.4]{Lang}).
The field $F$ is the quotient field of $A$.

\begin{lemma}\label{quadratic}
Let $P,P',Q,Q'\in A$ be such that $Q$ and $Q'$ are non-negative and $Q'$ is not
a square in $A$. If $P+\sqrt{Q}=P'+\sqrt{Q'}$, then $P=P'$ and $Q=Q'$.
\end{lemma}

\begin{proof}
We have
\[
Q=(P'-P+\sqrt{Q'})^2=(P'-P)^2 +2(P'-P)\sqrt{Q'}+Q'.
\]
If $P'\neq P$, then $\sqrt{Q'}\in F$, hence $Q'$ is a square in $F$ and therefore, $Q'$ is
a square in $A$, a contradiction. It follows that
$P'=P$ and hence $Q'=Q$.
\end{proof}

\begin{theorem}\label{f^m}
Let $f$ be a positive homogeneous function on $\R^n$, $n\geq 2$, of degree $1$ such that
$f(x)^m = P_m (x) + \sqrt{ Q_m(x)}$ for large odd $m$,
where $P_m$ are degree $m$ polynomials, and $Q_m$ are positive degree $2m$ polynomials.
Then $f(x)= P (x) + \sqrt{ Q(x)}$,
where $P$ is a linear polynomial and $Q$ is a positive quadratic polynomial.
\end{theorem}

\begin{proof}
We claim that all polynomials $Q_m$ are not squares in $A$ and $F$.
For if $Q_m=S_m^2$ for some odd $m$ and a polynomial $S_m$,
we have $f(x)^m = P_m (x) + |S_m(x)|$. As $n\geq 2$, there is a nontrivial zero $x$ of the
polynomial $P_m-S_m$. Changing sign if necessary, we may assume that $P_m(x)=S_m(x)\leq 0$.
Then $f(x)^m = P_m (x) - S_m(x)=0$.
But $f$ is positive, a contradiction. The claim is proved.

Let $m$ and $m'$ be two large odd integers. We have
\[
P_{mm'}  + \sqrt{ Q_{mm'}}=f^{mm'}=\left(P_m  + \sqrt{ Q_m}\right)^{m'}=K_m+L_m\sqrt{ Q_m}
\]
for some polynomials $K_m$ and $L_m$. By Lemma \ref{quadratic}, $Q_{mm'}=L_m^2 Q_m$. Similarly, $Q_{mm'}=L_{m'}^2 Q_{m'}$
for a polynomial $L_{m'}$. It follows that the ratio $Q_m/Q_{m'}$ is a square in
$F$ for all large odd $m$ and $m'$.
Moreover, there is a polynomial $D$ such that $Q_m=S_m^2 D$ for a polynomial $S_m$
for all large odd $m$. In particular, $D$ is not a square in $A$ and $F$.

Consider the ring $F(\sqrt{D})$ of all functions of the form $A+B\sqrt{D}$, where $A,B\in F$.
As $D$ is not a square in $F$, the ring $F(\sqrt{D})$ is a quadratic field extension of $F$ with $F$-basis $\{1,\sqrt{D}\}$.
Note that $f^m\in F(\sqrt{D})$ for all large odd $m$. It follows that $f\in F(\sqrt{D})$.

Write
\[
f=P + S\sqrt{D}
\]
for some rational functions $P,S\in F$.

Let $B$ be the set of all functions $s\in F(\sqrt{D})$ that are
integral over $A$, i.e., $s$ is a root of a monic
polynomial in one variable with coefficients in $A$.
By \cite[Chapter VII, Proposition 1.4]{Lang}, $B$ is a subring of $F(\sqrt{D})$ containing $A$. Since $\sqrt{D}\in B$, we have
\[
f^m=P_m  + \sqrt{ Q_m}=P_m  + S_m\sqrt{D}\in B
\]
for an odd $m$. It follows that $f\in B$.

Similarly, $\bar f:=P - S\sqrt{D}\in B$ since
\[
{\bar f}^m=(P - S\sqrt{D})^m=P_m  - S_m\sqrt{D}\in B.
\]
Therefore, the functions
$P=\frac{1}{2}(f+\bar f)$ and $S^2 D=\frac{1}{4}(f-\bar f)^2$ belong to $B$.
Since $A$ is a factorial ring, all rational functions in $F$ that are integral over $A$ belong to $A$,
i.e., $F\cap B=A$ by \cite[Chapter VII, Proposition 1.7]{Lang}. Thus,
$f=P+\sqrt{Q}$, where $P$ and $Q:=S^2 D$ are polynomials.

As $f$ is a positive homogeneous function of degree $1$, for every positive real number $a$,
\[
P(x)+\sqrt{Q(x)}=f(x)=af(a^{-1} x)=aP(a^{-1} x)+\sqrt{a^2Q(a^{-1} x)}
\]
for all $x$. By Lemma \ref{quadratic}, $P(x)=aP(a^{-1} x)$ and $Q(x)=a^2Q(a^{-1} x)$.
It follows that the polynomial $P$ is linear and $Q$ is quadratic.
\end{proof}


\begin{theorem} Let $n$ be an odd positive integer.
Let $K$ be an infinitely smooth convex body in
$\R^n,$ containing the origin as its interior point,
and suppose there exists a  number  $N\ge n$ such that for  every $\xi\in S^{n-1}$ and every $m\ge N$ we have
$$A_{K,\xi}^{(m)}(0)=0.$$
Then   $K$ is an ellipsoid.

\end{theorem}

\begin{proof}

By formula (\ref{A^(k-even)}), the Fourier transform of $\|x\|_K^{-n+m+1}+\| - x\|_K^{-n+m+1}$, for even $m\ge N$, is zero outside of the origin. The part supported in the origin is a derivative of the delta function, so $\|x\|_K^{-n+m+1}+\| - x\|_K^{-n+m+1}$ is a polynomial. Similarly, if $m\ge N$ is odd, then (\ref{A^(k-even)}) implies that $\|x\|_K^{-n+m+1} - \| - x\|_K^{-n+m+1}$  is also a polynomial.
Thus for any $k\ge (N-n+1)/2$ we have
$$\|x\|_K^{2k} + \|-x\|_K^{2k} = Q_1(x),$$
$$\|x\|_K^{2k+1} - \|-x\|_K^{2k+1} = Q_2(x),$$
$$\|x\|_K^{4k} + \|-x\|_K^{4k} = Q_3(x),$$
$$\|x\|_K^{4k+2} + \|-x\|_K^{4k+2} = Q_4(x),$$
where $Q_1$, $Q_2$, $Q_3$, $Q_4$ are homogeneous polynomials of degrees $2k$, $2k+1$, $4k$, $4k+2$ correspondingly.

Squaring the first two of these equalities and subtracting the second two equalities, we get
$$(\|x\|_K  \|-x\|_K) ^{2k} = P_1(x),$$
$$(\|x\|_K   \|-x\|_K) ^{2k+1} = P_2(x),$$
for some homogeneous polynomials $P_1$ and $P_2$.

Defining a star body $L$ by the formula $\|x\|_L = \sqrt{\|x\|_K   \|-x\|_K}$ and appying Lemma \ref{lem:product}, we obtain that  $L$ is an ellipsoid. Thus, there is a positive  quadratic polynomial $B$ such that
\begin{equation}\label{B}\|x\|_K   \|-x\|_K = B(x), \quad \forall x\in \R^n.
\end{equation}

Furthermore,  equation (\ref{B}) together with
$$\|x\|_K^{2k+1} - \|-x\|_K^{2k+1} = T_k(x),$$
 for every   $k\ge (N-n)/2$ and for some odd polynomials  $T_k$, yield a quadratic equation that gives
$$ \|x\|_K^{2k+1} = \frac12\left(T_k+\sqrt{T_k^2(x)+4 [B(x)]^{2k+1}}\right).$$

Using Theorem \ref{f^m}, we conclude that the Minkowski functional of $K$ is of the form
$$ \|x\|_K = P(x) + \sqrt{Q(x)},$$
where $P$ is a polynomial of degree 1, and $Q$ is a positive polynomial of degree 2. Thus, $K$ is an ellipsoid.

\end{proof}

\bigbreak

{\bf Acknowledgements.} The first named author was supported in part by the NSF grant
DMS-1265155. The second author has been supported by the NSF grant DMS-1160206. The third
author was supported in part by NSERC. We would like to thank Mark Agranovsky and Mikhail Zaidenberg for
bringing the problem to our attention.

\bibliographystyle{amsplain}

\bigskip

\bigskip

\noindent \textsc{Alexander \ Koldobsky}: Department of
Mathematics, University of Missouri, Columbia, MO 65211, USA

\smallskip

\noindent \textit{E-mail:} \texttt{koldobskiya@missouri.edu}

\bigskip

\noindent \textsc{Alexander S. Merkurjev}: Department of Mathematics, University of California, Los Angeles,
CA 90095-1555, USA

\smallskip

\noindent \textit{E-mail:} \texttt{merkurev@math.ucla.edu}

\bigskip

\noindent \textsc{Vladyslav \ Yaskin}: Department of Mathematical and Statistical Sciences, University of Alberta, Edmonton, Alberta T6G 2G1, Canada

\smallskip

\noindent \textit{E-mail:} \texttt{yaskin@ualberta.ca}

\end{document}